%% file: _main_DRAFT.tex
\documentclass[a4paper, 10pt]{article}

\usepackage{packageList}	
\usepackage{macros}

\graphicspath{{./Figures/}}
\pgfplotsset{compat=1.11}
\newlength\fwidth

\title{Spectral equivalence of unsymmetric kernel matrices and applications}
\author[1,2]{Tizian Wenzel\thanks{wenzel@math.lmu.de}}
\author[2]{Armin Iske}

\affil[1]{Department of Mathematics, Ludwig Maximilian University of Munich (Munich, Germany)}
\affil[2]{Department of Mathematics, Universität Hamburg (Hamburg, Germany)}

\newif\iflong			
\longfalse

\begin{document}

\maketitle %
  
\begin{abstract}
Symmetric kernel matrices are a well-researched topic in the literature of kernel based approximation.
In particular stability properties in terms of lower bounds on the smallest eigenvalue of such symmetric kernel matrices are thoroughly investigated, as they play a fundamental role in theory and practice.

%
In this work, we focus on unsymmetric kernel matrices 
and derive stability properties under small shifts 
by establishing a spectral equivalence to their unshifted, symmetric versions. 
This extends and generalizes results for translational invariant kernels upon Quak et al [SIAM Journal on Math.\ Analysis, 1993] and Sivakumar and Ward [Numerische Mathematik, 1993], 
however focusing instead on finitely smooth kernels.

As applications, we consider convolutional kernels over domains, which are no longer translational invariant, but which are still an important class of kernels for applications.
For these, we derive novel lower bounds for the smallest eigenvalue of the kernel matrices in terms of the separation distance of the data points,
and thus derive stability bounds in terms of the condition number.
\end{abstract}

\section{Introduction}
\label{sec:introduction}

Kernel based methods are well established techniques used in approximation theory, numerical analysis and machine learning
\cite{buhmann2000radial,fasshauer2007meshfree,wendland2005scattered}.

In this work we focus on strictly positive definite kernels, i.e.\ a symmetric functions $k: \Omega \times \Omega \rightarrow \R$ defined on $\Omega \subseteq \R^d$.
Given any set of $n \in \N$ pairwise distinct points $X = \{x_1, ..., x_n\}$, 
the kernel matrix $A_X := A_{k, X}$ collects all pairwise kernel evaluations, 
i.e.\ $A_{X} = (k(x_i, x_j))_{1\leq i, j \leq n} \in \R^{n \times n}$.
The kernel being strictly positive definite means, that the kernel matrix $A_X$ is positive definite for any set of pairwise distinct points $X \subset \Omega$,
and thus it holds $\lambda_{\min}(A_X) > 0$.

Strictly positive definite kernels can be used for various applications,
and the recovery of functions based on scattered data is among the most frequent ones:
Given a domain $\Omega$ and input points $X = \{ x_1, ..., x_n\} \subset \Omega$ with corresponding function values $f(x_1), ..., f(x_n)$ for some function $f \in \mathcal{C}(\Omega)$,
a kernel interpolant 
\begin{align}
\label{eq:kernel_approximant}
s_{f, X} = \sum_{j=1}^n \alpha_j k(\cdot, x_j)
\end{align}
can be computed based on the interpolation conditions $s_{f, X}(x_i) = f(x_i)$ for $i=1, ..., n$.
These interpolation conditions give rise to a linear equation system $A_{X} \alpha = f$ with $\alpha = (\alpha_1, ..., \alpha_n)^\top \in \R^n$, $f = (f(x_1), ..., f(x_n))^\top \in \R^n$.
This showcases the importance of the condition number $\text{cond}(A_{X}) = \frac{\lambda_{\max}(A_{X})}{\lambda_{\min}(A_{X})}$ of the kernel matrix for kernel based methods.
Thus the condition number $\text{cond}(A_{X}) $ and in particular the smallest eigenvalue $\lambda_{\min}(A_{X})$ were thoroughly investigated in the literature, see e.g.\ \cite{narcowich1991norms,narcowich1994condition,schaback1994lower,schaback1995error,diederichs2019improved}.
These estimates usually lower bound the smallest eigenvalue $\lambda_{\min}(A_{X})$ in terms of the separation distance $q_X := \frac{1}{2} \cdot \min_{x_i \neq x_j \in X} \Vert x_i - x_j \Vert$.

These works focus on translational invariant kernels that can be written as $k(x, z) = \Phi(x - z)$ for some function $\Phi: \R^d \rightarrow \R$ or even radial basis function kernels which can be expressed as $k(x, z) = \varphi(\Vert x - z \Vert)$ for some univariate function $\varphi: \R_{\geq 0} \rightarrow \R$.
The reason for this is, that for such kernels Fourier techniques can be employed by analyzing the Fourier transform ${\hat{\Phi}: \R^d \rightarrow \R}$.
For this, we express translational invariant kernels via its Fourier transform, i.e.\
\begin{align}
\label{eq:kernel_via_fourier}
k(x, z) = \Phi(x - z) = (2\pi)^{-d/2}  \int_{\R^d} \hat{\Phi}(\omega) e^{i\omega^\top (x-z)} ~ \mathrm{d}\omega
\end{align}
such that one obtains the formula
\begin{align}
\label{eq:kernel_vector_product}
\alpha^\top A_{X} \alpha &= (2\pi)^{-d/2} \int_{\R^d} \hat{\Phi}(\omega) \left \vert \sum_{j=1}^{n} \alpha_j e^{i \omega^\top x_j} \right \vert^2 ~ \mathrm{d}\omega,
\end{align}
for any $\alpha \in \R^n$.

In this work we will also be concerned with convolutional kernels $k^*: \Omega \times \Omega \rightarrow \R$ for Lipschitz domains $\Omega \subset \R^d$.
Such convolutional kernels can be defined for any kernel $k: \Omega \times \Omega \rightarrow \R$ as
\begin{align}
\label{eq:convKernel}
\begin{aligned}
k^*: \Omega \times \Omega &\rightarrow \R, \\
x, ~ z &\mapsto \int_\Omega k(x, y)k(y, z) ~ \mathrm{d}y.
\end{aligned}
\end{align}
The notion \textit{convolutional kernel} and the notation $k^*$ are motivated by the case of translational invariant kernels $k(x, z) = \Phi(x - z)$ over $\Omega = \R^d$,
because in this case Eq.~\eqref{eq:convKernel} can be written as a convolution of the base kernel $k$ with itself:
\begin{align*}
k^*(x, z) = \int_\Omega k(x, y)k(y, z) ~ \mathrm{d}y = \int_\Omega \Phi(x-y)\Phi(y-z) ~ \mathrm{d}y = (k * k)(x, z).
\end{align*}
Such convolutional kernels are important for applications, 
e.g.\ when reconstructions are done based on mean values, 
such as in Radon transform \cite{de2018image}, kernel-based finite volume methods \cite{aboiyar2010adaptive}, 
in signal theory \cite{plonka2018numerical} or in kernel Galerkin methods \cite{wendland1999meshless}.
Stability bounds on kernel matrices are also relevant for deriving inverse statements for kernel based approximation \cite{schaback2002inverse,wenzel2025sharp},
and in particular stability bounds for the convolutional kernel are likely important for proving inverse statements for the superconvergence case \cite{schaback2018superconvergence,karvonen2025general},
see also the discussion in \Cref{sec:background}.

In order to lower bound the smallest eigenvalue of the kernel matrix of convolutional kernels $k^*$,
we make use of integral discretization techniques to reduce the problem to lower bounding the smallest singular value of unsymmetric kernel matrices of the base kernel $k$ as introduced in \cite{quak1993least,sivakumar1993least}.
Working with Fourier analysis \cite{lowitzsch1996l2} instead of Schoenbergs characterizations of translational invariant kernels,
we present refined results that allow to express our bounds solely based on the separation distance of the data points.
In contrast to \cite{quak1993least,sivakumar1993least},
we focus on finitely smooth kernels, 
that are characterized by the decay of their Fourier transform.


The paper is organized as follows:
In \Cref{sec:background} we introduce further background information and related results.
Subsequently, in \Cref{sec:refined_bounds_unsymm} we prove a spectral equivalence of specific unsymmetric kernel matrices to standard kernel matrices.
This then allows one to derive stability bounds on the kernel matrix of convolutional kernels $k^*$ as formulated in \Cref{sec:stability_bounds_convkernel}.
\Cref{sec:numerics} illustrates the theoretical findings with numerical results, 
and \Cref{sec:outlook} finally concludes the paper and provides an outlook.


\section{Background}
\label{sec:background}

We start by collecting some necessary information about translational invariant kernels $k: \R^d \times \R^d \rightarrow \R$,
which can be written as
\begin{align}
\label{eq:transl_inv_kernel}
k(x, z) = \Phi(x - z)
\end{align}
and are usually analyzed with help of the Fourier transform $\hat{\Phi}$.
The properties of the kernel $k$ are then characterized by the decay of the function $\hat{\Phi}$.
If there are constants $c_\Phi, C_\Phi > 0$ such that it holds
\begin{align}
\label{eq:fourier_decay}
c_\Phi (1 + \Vert \omega \Vert^2)^{-\tau} \leq \hat{\Phi}(\omega) \leq C_\Phi (1+\Vert \omega \Vert^2)^{-\tau},
\end{align}
then the kernel $k$ is called a kernel of finite smoothness $\tau > d/2$.
In this case,
the reproducing kernel Hilbert space (RKHS) $\mathcal{H}_k(\R^d)$ of $k$ can be shown to be norm-equivalent to the Sobolev space $H^\tau(\R^d)$ of smoothness $\tau$,
and this equivalence can also be extended to Lipschitz domains $\Omega \subset \R^d$.
This connection motivates the notion of a Sobolev kernel.
The well known classes of Matérn kernels or Wendland kernels satisfy such a behaviour as in Eq.~\eqref{eq:fourier_decay}.

For the analysis of the stability of kernel based algorithms,
it is necessary to investigate the condition number of the kernel matrix $A_{k, X}$.
While the largest eigenvalue $\lambda_{\max}(A_{k, X})$ mainly scales according to size $|X|$ of the matrix $A_{k, X}$,
the smallest eigenvalue primarly depends on the smallest distance between the points $X$,
characterized in terms of the separation distance defined as 
\begin{align}
\label{eq:separation_distance}
q_X := \frac{1}{2} \min_{x_i \neq x_j \in X} \Vert x_i - x_j \Vert.
\end{align}
For translational invariant kernels satisfying Eq.~\eqref{eq:fourier_decay} and points $X$ within a bounded domain $\Omega$,
the smallest eigenvalue of $A_{k, X}$ can be bounded as 
\begin{align}
\label{eq:lower_bound_lambdamin}
\lambda_{\min}(A_{k, X}) \geq c_{\min} q_X^{2\tau - d}
\end{align}
for some constant $c_{\min}$ independent of $X$,
see e.g.\ \cite[Chapter 12.2]{wendland2005scattered}.

The convolutional kernel $k^*$ introduced in Eq.~\eqref{eq:convKernel} is a domain specific kernel,
as its definition relies on the domain $\Omega$.
Given that $k$ is strictly positive definite,
the strict positive definiteness of $k^*$ follows immediately from
\begin{align*}
\alpha^\top A_{k^*, X} \alpha = \int_\Omega \left| \sum_{j=1}^n \alpha_j k(y, x_j) \right|^2 ~ \mathrm{d}y > 0
\end{align*}
for $(\alpha_1, ..., \alpha_n) \neq 0$.
The reproducing kernel Hilbert space $\mathcal{H}_{k^*}(\Omega)$ of $k^*$ can be characterized with help of the integral operator
\begin{align}
\label{eq:int_operator}
\begin{aligned}
T: ~ L^2(\Omega) &\rightarrow L^2(\Omega) \\
\alpha &\mapsto \int_\Omega k(\cdot, y) \alpha(y) ~ \mathrm{d}y.
\end{aligned}
\end{align}
It can be shown, that the RKHS $\mathcal{H}_{k^*}(\Omega)$ is exactly the image of this integral operator $T_k$ on $L^2(\Omega)$ \cite{wendland2005scattered, steinwart2012mercer}, 
i.e.\ $TL^2(\Omega) = \mathcal{H}_{k^*}(\Omega)$.
This space $TL^2(\Omega)$ is known to play a crucial role in the analysis of kernel based superconvergence \cite{schaback1999improved, schaback2018superconvergence, karvonen2025general},
for which standard approximation rates can be roughly doubled.
Thus the analysis of the spectral properties of the convolutional kernel $k^*$ presented in \Cref{sec:stability_bounds_convkernel} is also expected to advance the understanding of this superconvergence phenomenon.
In particular,
the derivation of corresponding inverse statements for superconvergence likely requires stability bounds \cite{wenzel2025sharp}.


\section{Refined bounds for unsymmetric kernel matrices}
\label{sec:refined_bounds_unsymm}

In this section we start by deriving some utility results,
which we then leverage to show our first main result,
the spectral equivalence of specific unsymmetric kernel matrices to symmetric kernel matrices.




\subsection{Preliminary estimates}

We start by recalling a basic inequality that was also used in \cite{sivakumar1993least, quak1993least},
and which allows to lower bound singular values of matrices with help of their Rayleigh quotient.
For convenience, we also provide the short elementary proof.

For this, we introduce the symmetric part $A_+$ and the asymmetric part $A_-$ of a square matrix $A \in \R^{n \times n}$ as
\begin{align}
\label{eq:def_symm_unsymm_matrix}
A_+ := \frac{A + A^\top}{2}, ~ A_- := \frac{A - A^\top}{2} \quad \Rightarrow \quad A = A_+ + A_-.
\end{align}

\begin{prop}
\label{prop:estimate_symm_part}
Let $A \in \R^{n \times n}$. 
Then it holds
\begin{align*}
| \langle A_+ \alpha, \alpha \rangle | \leq \Vert A\alpha \Vert \cdot \Vert \alpha \Vert.
\end{align*}
\end{prop}

\begin{proof}

For the asymmetric part $A_-$ it holds
\begin{align}
\label{eq:asymmetric_part}
\langle A_- \alpha, \alpha \rangle = \frac{1}{2} \langle (A - A^\top) \alpha, \alpha \rangle = \frac{1}{2} \left( \langle A\alpha, \alpha \rangle - \langle \alpha, A\alpha \rangle \right) = 0.
\end{align}
Thus it holds by using Cauchy Schwarz inequality
\begin{align*}
| \langle A_+ \alpha, \alpha \rangle | = |\langle A\alpha, \alpha \rangle | \leq \Vert A\alpha \Vert \cdot \Vert \alpha \Vert.
\end{align*}
\end{proof}

As a preparatory result to prove our first main result on the spectral equivalence in \Cref{th:utility_theorem} in \Cref{subsec:rayleigh_bounds},
we prove the following utility result.
It allows to bound the effect of an additional $\sin \left( \frac{\omega^\top b}{2} \right)^2$ term within the integral of Eq.~\eqref{eq:kernel_vector_product}.
Recall that $A_X := A_{k,X}$ refers to the kernel matrix and $\alpha = (\alpha_1, ..., \alpha_n)^\top$.

\begin{prop}
\label{prop:estimate_integral_sinus_improved}
Let $\Omega \subset \R^d$ be bounded and $\tau > d/2$.
For any $\varepsilon \in (0, 1)$ there is a constant $c_\varepsilon > 0$ (only depending on $d, \tau, \Omega$ and $\varepsilon$) such that 
for all $X \subset \Omega$ of finitely many pairwise distinct points,
for all $\alpha = (\alpha_j)_{j=1}^{|X|}$ 
and all $b \in \R^d$ with $\Vert b \Vert \leq c_\varepsilon \cdot q_X$ 
($c_\varepsilon \propto \varepsilon^{\frac{1}{2} + \frac{1}{2\tau-d}} \leq \varepsilon^{1/2}$) it holds 
\begin{align}
\begin{aligned}
\label{eq:estimate_integral_sinus}
&\int_{\R^d} (1+\Vert \omega \Vert^2)^{-\tau} \left| \sum_{j=1}^n \alpha_j e^{i \omega^\top x_j}  \right|^2 \cdot \sin \left( \frac{\omega^\top b}{2} \right)^2 ~ \mathrm{d}\omega \\
&\qquad \leq~ 
2\varepsilon \cdot 
\int_{\R^d} (1+\Vert \omega \Vert^2)^{-\tau} \left| \sum_{j=1}^n \alpha_j e^{i \omega^\top x_j}  \right|^2 ~ \mathrm{d}\omega
\end{aligned}
\end{align}
For $\tau > 1$, 
it holds also the improved estimate
\begin{align}
\begin{aligned}
\label{eq:estimate_integral_sinus_improved}
(2\pi)^{-d/2} &\int_{\R^d} (1+\Vert \omega \Vert^2)^{-\tau} \left| \sum_{j=1}^n \alpha_j e^{i \omega^\top x_j}  \right|^2 \cdot \sin \left( \frac{\omega^\top b}{2} \right)^2 ~ \mathrm{d}\omega \\
\leq&~ 2\varepsilon \cdot c_{\min}^{1/\tau} \left( \frac{\alpha^\top A_X \alpha}{\Vert \alpha \Vert^2} \right)^{1-1/\tau} q_X^{2-d/\tau} \cdot \Vert \alpha \Vert^2.
\end{aligned}
\end{align}
using $\alpha^\top A_X \alpha := \int_{\R^d} (1+\Vert \omega \Vert^2)^{-\tau}  \left| \sum_{j=1}^n \alpha_j e^{i \omega^\top x_j}  \right|^2 ~ \mathrm{d}\omega$.
\end{prop}

Note that $\tau > d/2$ implies $2-d/\tau > 0$ and thus Eq.~\eqref{eq:estimate_integral_sinus_improved} is indeed an improvement over Eq.~\eqref{eq:estimate_integral_sinus} since 
$c_{\min} q_X^{2\tau-d} \leq \lambda_{\min}(A) \leq \frac{\alpha^\top A \alpha}{\Vert \alpha \Vert^2}$
by Eq.~\eqref{eq:lower_bound_lambdamin}.
%
%

\begin{proof}
We start by proving Eq.~\eqref{eq:estimate_integral_sinus}:
We make use of a localization estimate, which we take from \cite[Theorem 6]{wenzel2024novel} and can be found in a similar fashion in \cite[Lemma 3.3]{narcowich2006sobolev}:
For any $\varepsilon \in (0, 1)$,
there is a constant 
\begin{align}
\label{eq:constant_a}
a = a_{\tau, d, \Omega, \varepsilon} = c_{\tau, d, \Omega} \varepsilon^{-\frac{1}{2\tau-d}} > 0
\end{align}
such that for any $X \subset \Omega$ of finitely many pairwise distinct points and any $\alpha = (\alpha_j)_{j=1}^{|X|} \in \R^d$ we have for any $R \geq \frac{\max(a, \pi)}{q_X}$:
\begin{align*}
&(1-\varepsilon) \cdot \int_{\R^d} (1+\Vert \omega \Vert^2)^{-\tau} \cdot \left \vert \sum_{j=1}^{|X|} \alpha_j e^{i \omega^\top x_j} \right \vert^2 ~ \mathrm{d}\omega \notag \\
\leq~&
\int_{B_{2R}} (1+\Vert \omega \Vert^2)^{-\tau} \cdot \left \vert \sum_{j=1}^{|X|} \alpha_j e^{i \omega^\top x_j} \right \vert^2 ~ \mathrm{d}\omega, 
\end{align*}
which is equivalent to
\begin{align}
&\int_{\R^d \setminus B_{2R}} (1+\Vert \omega \Vert^2)^{-\tau} \left \vert \sum_{j=1}^{|X|} \alpha_j e^{i \omega^\top x_j} \right \vert^2 ~ \mathrm{d}\omega \notag \\
\leq~& \varepsilon \cdot \int_{\R^d} (1+\Vert \omega \Vert^2)^{-\tau} \left \vert \sum_{j=1}^{|X|} \alpha_j e^{i \omega^\top x_j} \right \vert^2 ~ \mathrm{d}\omega. \label{eq:localization_estimate2}
\end{align}
We decompose the integral $\int_{\R^d}$ from Eq.~\eqref{eq:estimate_integral_sinus} as $\int_{B_{2R}} + \int_{\R^d \setminus B_{2R}}$ using $R = \frac{\max(a, \pi)}{q_X}$.
\begin{enumerate}
\item For the first integral $\int_{B_{2R}}$ we use $\sin \left( \frac{\omega^\top b}{2} \right)^2 \leq \frac{\Vert \omega \Vert^2 \Vert b \Vert^2}{4} \leq R^2 \Vert b \Vert^2$ to obtain
\begin{align*}
&\int_{B_{2R}} (1+\Vert \omega \Vert^2)^{-\tau}  \left| \sum_{j=1}^n \alpha_j e^{i \omega^\top x_j}  \right|^2 \cdot \sin \left( \frac{\omega^\top b}{2} \right)^2 ~ \mathrm{d}\omega \\
\leq~& R^2 \Vert b \Vert^2 \cdot \int_{B_{2R}} (1+\Vert \omega \Vert^2)^{-\tau} \left| \sum_{j=1}^n \alpha_j e^{i \omega^\top x_j}  \right|^2\mathrm{d}\omega \\
\leq~& \frac{\max(a, \pi)^2 \Vert b \Vert^2}{q_X^2} \cdot \int_{\R^d} (1+\Vert \omega \Vert^2)^{-\tau} \left| \sum_{j=1}^n \alpha_j e^{i \omega^\top x_j}  \right|^2\mathrm{d}\omega.
\end{align*}
\item For the second integral $\int_{\R^d \setminus B_{2R}}$ we simply estimate $\sin \left( \frac{\omega^\top b}{2} \right)^2 \leq 1$ and subsequenty make use of Eq.~\eqref{eq:localization_estimate2}:
\begin{align*}
&\int_{\R^d \setminus B_{2R}} (1+\Vert \omega \Vert^2)^{-\tau} \left| \sum_{j=1}^n \alpha_i e^{i \omega^\top x_j}  \right|^2 \cdot \sin \left( \frac{\omega^\top b}{2} \right)^2 ~ \mathrm{d}\omega \\
\leq&~\int_{\R^d \setminus B_{2R}} (1+\Vert \omega \Vert^2)^{-\tau}  \left| \sum_{j=1}^n \alpha_i e^{i \omega^\top x_j}  \right|^2 ~ \mathrm{d}\omega \\
\leq&~ \varepsilon \cdot \int_{\R^d} (1+\Vert \omega \Vert^2)^{-\tau}  \left| \sum_{j=1}^n \alpha_i e^{i \omega^\top x_j}  \right|^2 ~ \mathrm{d}\omega 
\end{align*}
\end{enumerate}
Putting these two estimates together, we obtain
\begin{align*}
&\int_{\R^d} (1+\Vert \omega \Vert^2)^{-\tau} \left| \sum_{j=1}^n \alpha_j e^{i \omega^\top x_j}  \right|^2 \cdot \sin \left( \frac{\omega^\top b}{2} \right)^2 ~ \mathrm{d}\omega \\
\leq&~ \left( \frac{\max(a, \pi)^2 \Vert b \Vert^2}{q_X^2} + \varepsilon \right) \int_{\R^d} (1+\Vert \omega \Vert^2)^{-\tau} \left| \sum_{j=1}^n \alpha_j e^{i \omega^\top x_j}  \right|^2 ~ \mathrm{d}\omega \\
\end{align*}
using $\Vert b \Vert \leq \frac{\sqrt{\varepsilon}}{\max(a, \pi)} \cdot q_X =: c_\varepsilon q_X$, 
the prefactor can be estimated with $\varepsilon + \varepsilon = 2 \varepsilon$,
which finishes the proof.
Note that $a \propto \varepsilon^{-\frac{1}{2\tau-d}}$ according to Eq.~\eqref{eq:constant_a},
i.e.\ $c_\varepsilon \propto \varepsilon^{\frac{1}{2} + \frac{1}{2\tau-d}}$ for $\varepsilon \rightarrow 0$.


The proof of Eq.~\eqref{eq:estimate_integral_sinus_improved} proceeds in the same way,
but the integral $\int_{B_{2R}}$ is estimated more accurately by decomposing it into $\int_{B_{r}} + \int_{B_{2R} \setminus B_r}$ for $r \in (0, 2R)$ and optimizing over $r$.
The details of the computation are given in \Cref{sec:outsourced_proof}.
\end{proof}

\subsection{Spectral equivalence for unsymmetric kernel matrices}
\label{subsec:rayleigh_bounds}

With this proposition at hand, 
we can state and prove the main result \Cref{th:utility_theorem} of this section,
which allows to compare the Rayleigh quotient of symmetric and unsymmetric kernel matrices.
For this, we use the notation 
\begin{align*}
k(X, X) :=& (k(x_i, x_j))_{1 \leq i, j \leq n} \\
\text{ and~ } k(X+b, X) :=& (k(x_i + b, x_j)_{1 \leq i, j \leq n}
\end{align*}
to refer to kernel matrices, 
and in particular we have set $X + b := \{x_1 + b, ..., x_n + b\}$.
After the proof of \Cref{th:utility_theorem}, 
its meaning as well as the improvement over related results in \cite{sivakumar1993least, quak1993least} is discussed.


\begin{theorem}[Main result 1]
\label{th:utility_theorem}
Let $\Omega \subset \R^d$ be bounded and let $k(x, z) = \Phi(x -z)$ be a translational invariant kernel that satisfies Eq.~\eqref{eq:fourier_decay} for some $\tau > d/2$.
Then there is a constant $c = c_{\tau, d, k, \Omega}$ only depending on $\tau, d, k, \Omega$,
such that for any $X \subset \Omega$ of pairwise distinct points and any $0 \neq b \in \R^d$ such that $\Vert b \Vert < c \cdot q_X$,
it holds for all $0 \neq \alpha \in \R^{|X|}$:
\begin{align}
\label{eq:intermediate_eq4b}
\frac{3}{4} \cdot \frac{\langle k(X, X) \alpha, \alpha \rangle}{\Vert \alpha \Vert^2} ~\leq~ \frac{\langle k(X+b, X)_+ \alpha, \alpha \rangle}{\Vert \alpha \Vert^2} ~<~ \frac{\langle k(X, X) \alpha, \alpha \rangle}{\Vert \alpha \Vert^2}.
\end{align}
For $\tau > 1$, this can be specialized to
\begin{align}
\begin{aligned}
\label{eq:intermediate_eq4a}
\frac{\langle k(X, X) \alpha, \alpha \rangle}{\Vert \alpha \Vert^2} - 
\left( \frac{\langle k(X, X) \alpha, \alpha \rangle}{\Vert \alpha \Vert^2} \right)^{1-1/\tau} q_X^{2-d/\tau} \\
\leq~ \frac{\langle k(X+b, X)_+ \alpha, \alpha \rangle}{\Vert \alpha \Vert^2} ~<~ \frac{\langle k(X, X) \alpha, \alpha \rangle}{\Vert \alpha \Vert^2}.
\end{aligned}
\end{align}

\end{theorem}
The benefit of Eq.~\eqref{eq:intermediate_eq4a} over \eqref{eq:intermediate_eq4b} is the smaller gap in the estimate for specific choices of $\alpha$.


\begin{proof}
Using \Cref{prop:estimate_symm_part} and following the simple algebraic calculations from the proof of \cite[Theorem 3.2]{quak1993least} (for a detailed calculation, see \Cref{sec:symm_kernel_part}) we have:
\begin{align}
\label{eq:intermediate_eq3}
(2\pi)^{d/2} \cdot \langle k(X+b, X)_+ \alpha, \alpha \rangle 
=& ~ \int_{\R^d} \hat{\Phi}(\omega) \left| \sum_{i=1}^n \alpha_i e^{i\omega^\top x_i} \right|^2 ~ \mathrm{d}\omega \\
&\quad - 2 \int_{\R^d} \hat{\Phi}(\omega) \left| \sum_{i=1}^n \alpha_i e^{i\omega^\top x_i} \right|^2 \sin \left( \frac{\omega^\top b}{2} \right)^2  ~ \mathrm{d}\omega \notag 
\end{align}
Thus we obtain the upper bound immediately as
\begin{align*}
&~ (2\pi)^{d/2} \cdot \langle k(X+b, X)_+ \alpha, \alpha \rangle \notag \\
\leq& \int_{\R^d} \hat{\Phi}(\omega) \left| \sum_{i=1}^n \alpha_i e^{i\omega^\top x_i} \right|^2 ~ \mathrm{d}\omega 
= (2\pi)^{d/2} \cdot \langle k(X, X)\alpha, \alpha \rangle.
\end{align*}
For the lower bound, we need to estimate the negative part from Eq.~\eqref{eq:intermediate_eq3},
which will be done with \Cref{prop:estimate_integral_sinus_improved} and Eq.~\eqref{eq:fourier_decay}.
Thus we rearrange Eq.~\eqref{eq:intermediate_eq3} and apply Eq.~\eqref{eq:estimate_integral_sinus_improved}
for $\Vert b \Vert \leq c_\varepsilon \cdot q_X$ and $c_\varepsilon$ according to \Cref{prop:estimate_integral_sinus_improved}, 
which gives
\begin{align*}
&~ \langle \alpha, k(X, X) \alpha \rangle - \langle \alpha, k(X+b, X)_+ \alpha \rangle \\
\leq&~ \frac{2 C_\Phi}{(2\pi)^{d/2}} \int_{\R^d} (1+\Vert \omega \Vert^2)^{-\tau} \left| \sum_{j=1}^n \alpha_j e^{i\omega^\top x_j} \right|^2 \cdot \sin \left( \frac{\omega^\top b}{2} \right)^2 ~ \mathrm{d}\omega \\
\leq&~ 4C_\Phi \varepsilon c_{\min}^{1/\tau} \left( \frac{\alpha^\top A \alpha}{\Vert \alpha \Vert^2} \right)^{1-1/\tau} q_X^{2-d/\tau} \cdot \Vert \alpha \Vert^2
\end{align*}
in the case $\tau > 1$.
Choosing $\varepsilon$ correspondingly, this already gives the first result Eq.~\eqref{eq:intermediate_eq4a}.
The same rearrangement of Eq.~\eqref{eq:intermediate_eq3}, 
but applying Eq.~\eqref{eq:estimate_integral_sinus} then yields Eq.~\eqref{eq:intermediate_eq4b}, which works for any $\tau > d/2$.
\end{proof}


From Eq.~\eqref{eq:asymmetric_part} we have $\langle k(X+b, X)\alpha, \alpha \rangle = \langle k(X+b, X)_+ \alpha, \alpha \rangle$,
thus \Cref{th:utility_theorem} relates the Rayleigh quotient for the matrices $k(X, X)$ and $k(X+b, X)$ to each other and in particular shows the spectral equivalence of $k(X, X)$ and $k(X+b, X)$.
This thus allows for statements on the alignment of the eigenvectors as discussed in \cite{wenzel2025spectral}:
Due to the two-sided bound of Eq.~\eqref{eq:intermediate_eq3}, 
the eigenvector to a small eigenvalue of $k(X, X)$ yields a small value of the Rayleigh quotient of $k(X+b, X)$, and vice versa.
In particular, up to the factor $\frac{3}{4}$ the (symmetric part of the) unsymmetric kernel matrix $k(X+b, X)$ enjoys the same stability estimates as $k(X, X)$, 
i.e.\ Eq.~\eqref{eq:lower_bound_lambdamin}.
Note that by picking a smaller constant $c = c_{\tau, d, \Omega}$, 
the prefactor $\frac{3}{4}$ within Eq.~\eqref{eq:intermediate_eq4b} can even be chosen closer to $1$.

Note that the two-sided bounds of \Cref{th:utility_theorem} heavily rely on having the same shift $b \in \R^d$ for every data point.
In case of having different shifts per data point, i.e.\ $\{x_1 + b_1, ..., x_n + b_n \}$, 
one rather needs to focus on singular value decomposition based equivalences,
and we leave this more intricate investigation of unsymmetric kernel matrices for future research.

Finally we comment on the relation of \Cref{th:utility_theorem} to related results in previous research \cite{quak1993least}.
Despite the analysis in \cite{quak1993least} is based on Schoenbergs characterization instead of Fourier analysis, the calculations can be compared:
\cite[Theorem 3.2]{quak1993least} also considered the unsymmetric kernel matrix $k(X+b, X)$,
however applied the worst case estimate $| \sum_{i=1}^n \alpha_i e^{i \omega^\top x_i} |^2 \leq N \Vert \alpha \Vert^2$ to bound the negative part of Eq.~\eqref{eq:intermediate_eq3}, see \cite[Remark 3.5]{quak1993least}.
Only for the Gaussian kernel \cite[Proposition 3.11]{quak1993least} and for the univariate case \cite[Section 5]{quak1993least}, they derive improved results.
Overall the case of finitely smooth kernel satisfying Eq.~\eqref{eq:fourier_decay} was not treated in \cite{quak1993least}.

\section{Stability bounds for convolutional kernels}
\label{sec:stability_bounds_convkernel}

The main result of this section is formulated in \Cref{th:main_result}, 
namely a lower bound on the Rayleigh quotient of the kernel matrix of the convolutional kernel $k^*$.
Subsequently, some corollaries formulate the result for some special cases of frequently used kernels, 
see \Cref{cor:conclusion_1}.

\begin{theorem}[Main result 2]
\label{th:main_result}
Let $\Omega \subset \R^d$ be a bounded domain and let $k: \Omega \times \Omega \rightarrow \R$ be a translational-invariant kernel. 
Then there is a constant $c = c_d > 0$ (only depending on the dimension $d$),
such that for any set $X \subset \Omega$ of $n := |X|$ pairwise distinct points it holds
\begin{align*}
\frac{\alpha^\top k^*(X, X) \alpha}{\Vert \alpha \Vert^2} 
&\geq c q^d \cdot \left( \frac{\alpha^\top k(X, X) \alpha}{\Vert \alpha \Vert^2} \right)^2 \qquad \qquad \forall 0 \neq \alpha \in \R^{|X|}
\end{align*}

with $q := \min(\mathrm{dist}(X, \partial \Omega), q_X) > 0$.
\end{theorem}

\begin{proof}
The proof follows the ideas of the integral discretization technique of \cite[Section 4]{quak1993least}, 
and in the end leverages \Cref{th:utility_theorem}.
We start by using \cite[Lemma 4.3]{quak1993least} for balls instead of cubes to obtain
\begin{align*}
\alpha^\top k^*(X, X) \alpha &\geq \sum_{i=1}^n \int_{B_q(x_i)} \left | \sum_{j=1}^n \alpha_j k(y, x_j) \right|^2 ~ \mathrm{d}y.
\end{align*}

By definition of the integral, these $n$ integrals can be approximated arbitrarily accurate with help of quadrature points:
Thus let $\{ b_\ell \}_{\ell = 1}^M \subset B_q(0)$ such that
for all $i=1, ..., n$
\begin{align*}
&\left| \int_{B_q(x_i)} \left | \sum_{j=1}^n \alpha_j k(y, x_j) \right|^2 ~ \mathrm{d}y - \frac{|B_q(x_i)|}{M} \sum_{\ell=1}^M \left| \sum_{j=1}^n \alpha_j k(x_j + b_\ell, x_j) \right|^2 \right| \\
\leq~& \frac{1}{2} \cdot \min_{i=1, ..., n} \int_{B_q(x_i)} \left | \sum_{j=1}^n \alpha_j k(y, x_j) \right|^2 ~ \mathrm{d}y,
\end{align*}
from which we directly obtain 
\begin{align*}
\frac{|B_q(x_i)|}{2M} \sum_{\ell=1}^M \left| \sum_{j=1}^n \alpha_j k(x_j + b_\ell, x_j) \right|^2 \leq \int_{B_q(x_i)} \left | \sum_{j=1}^n \alpha_j k(y, x_j) \right|^2 ~ \mathrm{d}y.
\end{align*}
This then allows for further lower estimate $\alpha^\top k^*(X, X) \alpha$ as 
\begin{align}
\label{eq:intermediate_eq1}
\alpha^\top k^*(X, X) \alpha &\geq \sum_{i=1}^n \int_{B_q(x_i)} \left | \sum_{j=1}^n \alpha_j k(y, x_j) \right|^2 ~ \mathrm{d}y \notag \\
&\geq \sum_{i=1}^n \frac{|B_q(x_i)|}{2M} \sum_{\ell=1}^M \left| \sum_{j=1}^n \alpha_j k(x_i + b_\ell, x_j) \right|^2 \notag \\
&= \frac{|B_q(0)|}{2} \sum_{\ell = 1}^M \frac{1}{M} \sum_{i=1}^n \left | \sum_{j=1}^n \alpha_j k(x_i + b_\ell, x_j) \right|^2 \notag \\
&\stackrel{\exists \ell^*}{\geq} \frac{|B_q(0)|}{2} \sum_{i=1}^n \left | \sum_{j=1}^n \alpha_j k(x_i + b_{\ell^*}, x_j) \right|^2. 
\end{align}
In the final step, the pigeonhole principle was leveraged: 
Within the sum $\sum_{\ell = 1}^M$, there needs to be at least one summand which is smaller that the average $\frac{1}{M} \sum_{\ell=1}^M$.

Eq.~\eqref{eq:intermediate_eq1} can be conveniently written as a matrix-vector-product:
Using $k(X+b, X) := (k(x_i + b_{\ell^*}, x_j))_{1 \leq i, j \leq n}$ as matrix notation, it holds
\begin{align}
\label{eq:matrix_vector_expression}
\alpha^\top k^*(X, X) \alpha
&\geq \frac{|B_q(0)|}{2} \sum_{i=1}^n \left | \sum_{j=1}^n \alpha_j k(x_i + b_{i, \ell^*}, x_j) \right|^2 \notag \\
&= \frac{|B_q(0)|}{2} \cdot \Vert k(X+b, X) \alpha \Vert^2.
\end{align}
Here, $k(X+b, X)$ is a non-symmetric square kernel matrix.
We continue by applying \Cref{prop:estimate_symm_part} using $A = k(X+b, X)$, such that we obtain
\begin{align*}
\alpha^\top k^*(X, X) \alpha
&\geq \frac{|B_q(0)|}{2} \cdot \frac{| \langle k(X+b, X)_+ \alpha, \alpha \rangle |^2}{\Vert \alpha \Vert^2}.
\end{align*}
For this expression, 
\Cref{th:utility_theorem} can be applied to lower bound the right hand side quantity further as
\begin{align*}
\alpha^\top k^*(X, X) \alpha
&\geq \frac{9|B_q(0)|}{32} \cdot \frac{(\alpha^\top k(X, X) \alpha)^2}{\Vert \alpha \Vert^2}.
\end{align*}
Dividing by $\Vert \alpha \Vert^2$ and using $|B_q(0)| = \frac{\pi^{d/2}}{\Gamma(d/2+1)} q^d$ gives the statement.
\end{proof}

A related result was also derived in \cite[Theorem 4.6]{quak1993least}, 
however involving an undesired additional factor $\varepsilon = \varepsilon(d, q, k, n) \in (0, 1]$ on the right hand side which weakened the estimate significantly.
In our result, we removed the dependency on $q$ and $n$, 
thus ending up with the constant $c = c_d$ only depending on the dimension.
Only for the case of the Gaussian kernel, \cite[Theorem 4.10]{quak1993least} was able to derive an improved statement, 
which is now complemented by our \Cref{th:main_result} for kernels of finite smoothness.

Based on \Cref{th:main_result}, 
we can directly conclude a lower bound on the smallest eigenvalue of kernel matrices for convolutional kernels.
This is stated in the following \Cref{cor:conclusion_1}:

\begin{cor}
\label{cor:conclusion_1}
Let $\Omega \subset \R^d$ be a bounded domain and let $k(x, z) = \Phi(x - z)$ be a translational invariant kernel that satisfies Eq.~\eqref{eq:fourier_decay}.
Then there is a constant $c = c_{d, k, \Omega} > 0$ (only depending on the dimension $d$, kernel $k$ and domain $\Omega$),
such that for any set $X \subset \Omega$ of $n := |X|$ pairwise distinct points it holds
\begin{align}
\begin{aligned}
\label{eq:lower_bound_conv}
\lambda_{\min}(k^*(X, X))
&\geq c q^d \cdot \lambda_{\min}(k(X, X))^2 \\
&\geq c q^{4\tau-d}
\end{aligned}
\end{align}
with $q := \min(\mathrm{dist}(X, \partial \Omega), q_X) > 0$.
\end{cor}

Such lower bounds on the smallest eigenvalue directly imply stability bounds in terms of the condition number,
and are thus of importance for applications \cite{de2018image,aboiyar2010adaptive,plonka2018numerical}.
Furthermore, they also contribute to the analysis of inconsistency of kernel interpolation in statistical learning theory \cite{haas2023mind}.

\begin{rem}
We believe that the lower bound from Eq.~\eqref{eq:lower_bound_conv} actually holds independent of the distance $\mathrm{dist}(X, \partial \Omega)$ of the points to the boundary,
when we assume some additional regularity on $\Omega$ such as an interior cone condition.
We leave an improvement of the statement in this direction for future research.
\end{rem}

\Cref{cor:conclusion_1} allows to obtain the following bound on the condition number of the kernel matrix:

\begin{cor}
\label{cor:conclusion_2}
Under the assumptions of \Cref{cor:conclusion_1}, there exists a constant $c>0$
(only depending on the dimension $d$, kernel $k$ and domain $\Omega$),
such that the condition number of $k^*(X, X)$ can be bounded as
\begin{align*}
\mathrm{cond}(k^*(X, X)) \leq c \cdot q^{-4\tau}.
\end{align*}
\end{cor}
\begin{proof}
It holds $\mathrm{cond}(k^*(X, X)) = \lambda_{\max}(k^*(X, X)) / \lambda_{\min}(k^*(X, X))$.
The smallest eigenvalue $\lambda_{\min}(k^*(X, X))$ can be bounded from below via Eq.~\eqref{eq:lower_bound_conv}.
The largest eigenvalue $\lambda_{\max}(k^*(X, X))$  can be bounded by $|X| \cdot \max_{x \in \Omega} |k^*(x, x)| \leq |X| \cdot |\Omega| \cdot \max_{x \in \Omega} |k(x, x)|$.
Based on volume comparison arguments \cite[Section 14.1]{wendland2005scattered},
the number of points $|X|$ can be related to the separation distance as $|X| \leq c' \cdot q_X^{-d} \leq c' q^{-d}$.
\end{proof}



%


\section{Numerical results}
\label{sec:numerics}

\subsection{Spectral equivalence}

In this section we visualize the spectral equivalence result stated in \Cref{th:utility_theorem}.
For both cases $d = 2$ and $d=3$ we consider $n=50$ well-distributed Halton points $X \in [0, 1]^d$
and a small perturbation vector $b \in \R^d$ such that $\Vert b \Vert = 0.1 q_X$,
we compute the kernel matrices $k(X, X)$ and $k(X+b, X)_+$  as in the setting of \Cref{th:utility_theorem}.
For $d=2$ we make use of the linear Matérn kernel defined by $\Phi(x) = \varphi(\Vert x \Vert)$ and $\varphi(r) = (1+r)\exp(-r)$,
and for $d=3$ we make use of the quadratic Matérn kernel given by $\varphi(r) = (3 + 3r + r^2)\exp(-r)$.

The spectral equivalence of Eq.~\eqref{eq:intermediate_eq4b} stated using Loewner order is
\begin{align*}
\frac{3}{4} k(X, X)  \leq~ k(X+b, X)_+ <~ k(X, X)
\end{align*}
and can be reformulated equivalently as
\begin{align}
\label{eq:equivalence_identity}
\frac{3}{4} \cdot I  \leq k(X, X)^{-1/2} k(X+b, X)_+ k(X, X)^{-1/2} <~ I.
\end{align} 
To visualize this spectral equivalance, 
\Cref{fig:spectral_equivalence} thus depicts absolute values of the matrix $k(X, X)^{-1/2} k(X+b, X)_+ k(X, X)^{-1/2}$.
One can observe that the diagonal values are close to one, 
while the off-diagonal values are close to zero (mind the logarithmic scaling),
which is in accordance with the theoretical result Eq.~\eqref{eq:equivalence_identity}.


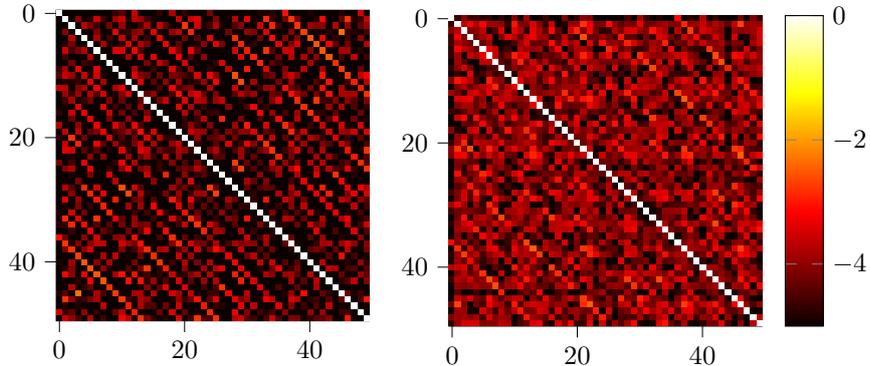
\begin{figure}[t]
\begin{center}
\setlength\fwidth{.47\textwidth}
\begin{minipage}{0.4\textwidth}
    \input{Figures/heatmap_identity_2.tex}
\end{minipage}%
\hspace{.2cm}
\begin{minipage}{0.48\textwidth}
\centering
\vspace{.1cm}
    \input{Figures/heatmap_identity_3.tex}
\end{minipage}
\end{center}
\caption{Heatmap visualization of the absolute values of the matrix $k(X, X)^{-1/2} k(X+b, X)_+ k(X, X)^{-1/2}$ using 50 well distributed Halton points $ X \subset [0, 1]^2$ (left, using linear Matérn kernel) respective $X \subset [0, 1]^3$ (right, using quadratic Matérn kernel).}
\label{fig:spectral_equivalence}
\end{figure}

\subsection{$\lambda_{\min}$ of convolutional kernels}

In this example,
we consider the basic and the linear Matérn kernel on $\Omega = [0, 1]$.
They are translational invariant kernel $k(x, z) = \Phi(x-z)$ with
\begin{align*}
\Phi_\text{basic}(r) &= \exp(-r) \\
\Phi_{\text{lin}}(r) &= (1+r)\exp(-r),
\end{align*}
in particular these kernels satisfy the Fourier decay condition of Eq.~\eqref{eq:fourier_decay} with $\tau_\text{basic} = 1$, $\tau_\text{linear} = 2$.

For both kernels, we consider each $n$ linearly spaces points $X_n \subset \Omega$ for 30 logarithmically equally spaced values of $n \in [10, 1000]$ and compute the smallest eigenvalue $\lambda_{\min}(k(X_n, X_n))$ as well as $\lambda_{\min}(k^*(X_n, X_n))$ of the kernel matrices for the kernel $k$ respective its convolutional kernel $k^*$.
\Cref{fig:lambda_min_matern} depicts the results, additionally showing a dashed line that indicates the asymptotic given by the lower bounds of Eq.~\eqref{eq:lower_bound_lambdamin} and Eq.~\eqref{eq:lower_bound_conv}. 
One can clearly observe that the smallest eigenvalue behaves according to the bounds derived in \Cref{cor:conclusion_1}.
In particular we mention that the results differ only marginally, if $\{ 0, 1\} \in X_n$.
This suggests, that $q \geq \min(\mathrm{dist}(X, \partial \Omega), q_X)$ is actually not necessary in \Cref{cor:conclusion_1},
but $q = q_X$ is sufficient.


\begin{figure}[t]
\setlength\fwidth{.55\textwidth}
\input{Figures/min_eigenvalue_matern_0.tex}
\input{Figures/min_eigenvalue_matern_1.tex}
\caption{Visualization of the smallest eigenvalue $\lambda_{\min}$ for the basic Matérn kernel (left) and the linear Matérn kernel (right) and their convolutional kernels $k^*$ over the number of equidistant interpolation points ($x$-axis).
The black dashed lines indicate the lower bounds.}
\label{fig:lambda_min_matern}
\end{figure}
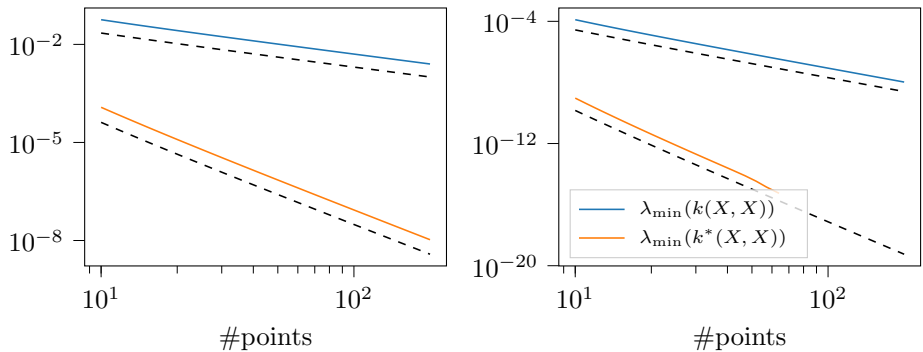

\section{Conclusion and outlook}
\label{sec:outlook}

In this paper we extended upon work from \cite{sivakumar1993least, quak1993least},
focusing on the case of finitely smooth kernels satisfying an algebraic decay of the Fourier transform as in Eq.~\eqref{eq:fourier_decay}.
The first main result \Cref{th:utility_theorem} provides two-sided bounds on the Rayleigh quotient for unsymmetric kernel matrices, 
showing an equivalence to the symmetric kernel matrices for small shifts.
The second main result \Cref{th:main_result} then provides lower bounds on the Rayleigh quotient for convolutional kernels, 
thus enabling stability bounds for convolutional kernels.

For future research, we aim at a better understanding of the relation between the unsymmetric kernel matrix $k(X+b, X)$ to the symmetric kernel matrix $k(X, X)$ for the case of different shifts $b_i$ per point. \\


\textbf{Acknowledgements:}
The work was mainly conducted when the first author was affiliated with the University of Hamburg.
The authors acknowledge financial support through the projects LD-SODA of the {\em Landesforschungsf\"orderung Hamburg} (LFF) and support from the RTG~2583 ``Modeling, Simulation and Optimization of Fluid Dynamic Applications''
funded by the {\em Deutsche Forschungsgemeinschaft} (DFG).

\IfFileExists{/home/wenzel/references.bib}{
\bibliography{/home/wenzel/references}				
}{
\bibliography{/home/math/wenzel/references}			
}
\bibliographystyle{abbrv}

\appendix

\section{Details on proof of \Cref{prop:estimate_integral_sinus_improved}}
\label{sec:outsourced_proof}

\begin{proof}
We split the integral of interest as $\int_{\R^d} = \int_{B_{2R}} + \int_{\R^d \setminus B_{2R}}$ and estimate separately:
\begin{itemize}
\item For the integral $\int_{B_{2R}}$, 
we start by estimating $\sin \left( \frac{\omega^\top b}{2} \right)^2 \leq \frac{\Vert b \Vert^2}{4} \cdot \Vert \omega \Vert^2$.
We keep the prefactor $\frac{\Vert b \Vert^2}{4}$ in mind and continue simply with $\Vert \omega \Vert^2$:
For a value $r \in (0, 2R)$ we thus consider
\begin{align*}
&~ \int_{B_{2R}} \Vert \omega \Vert^2 \cdot (1 + \Vert \omega \Vert^2)^{-\tau} \cdot \left| \sum_{j=1}^n \alpha_j e^{i \omega^\top x_j} \right|^2 ~ \mathrm{d}\omega \\
=& ~ \int_{B_r} \Vert \omega \Vert^2 \cdot (1 + \Vert \omega \Vert^2)^{-\tau} \cdot \left| \sum_{j=1}^n \alpha_j e^{i \omega^\top x_j} \right|^2 ~ \mathrm{d}\omega \\
&\qquad
+ \int_{B_{2R} \setminus B_r} \Vert \omega \Vert^2 \cdot (1 + \Vert \omega \Vert^2)^{-\tau} \cdot \left| \sum_{j=1}^n \alpha_j e^{i \omega^\top x_j} \right|^2 ~ \mathrm{d}\omega \\
\leq& ~ r^2 \cdot \int_{B_r} (1 + \Vert \omega \Vert^2)^{-\tau} \cdot \left| \sum_{j=1}^n \alpha_j e^{i \omega^\top x_j} \right|^2 ~ \mathrm{d}\omega \\ 
&\qquad
+ \max_{\tilde{r} \in [r, 2R]} \frac{\tilde{r}^2}{(1+\tilde{r}^2)^\tau} \cdot \int_{B_{2R} \setminus B_r} \left| \sum_{j=1}^n \alpha_j e^{i \omega^\top x_j} \right|^2 ~ \mathrm{d}\omega \\
\leq& ~ r^2 \cdot \int_{B_{2R}} (1 + \Vert \omega \Vert^2)^{-\tau} \cdot \left| \sum_{j=1}^n \alpha_j e^{i \omega^\top x_j} \right|^2 ~ \mathrm{d}\omega \\
&\qquad
+ \max_{\tilde{r} \in [r, 2R]} \frac{\tilde{r}^2}{\tilde{r}^{2\tau}} \cdot \int_{B_{2R}} \left| \sum_{j=1}^n \alpha_j e^{i \omega^\top x_j} \right|^2 ~ \mathrm{d}\omega \\
\leq& ~ r^2 \cdot \int_{B_{2R}} (1 + \Vert \omega \Vert^2)^{-\tau} \cdot \left| \sum_{j=1}^n \alpha_j e^{i \omega^\top x_j} \right|^2 ~ \mathrm{d}\omega 
+ r^{-2(\tau-1)} \cdot c_2 (2R)^d \Vert \alpha \Vert^2 \\
=:&~ r^2 \Psi_1 + r^{-2(\tau - 1)} \Psi_2.
\end{align*}
For the second part we used $\max_{\tilde{r} \in [r, 2R]} \tilde{r}^{2(1-\tau)} = r^{-2(\tau-1)}$ (recall $\tau > 1)$ as well as \cite[Theorem 3]{wenzel2025spectral} 
in the last inequality.
Note that for applying \cite[Theorem 3]{wenzel2025spectral}, $2R$ needs to be large enough, namely $2R \geq \frac{c_0}{q_X}$. 

This final expression can now be minimized in $r \in (0, 2R)$. 
A standard optima calculation yields the choice $r^* := \left( \frac{ (\tau-1) \Psi_2 }{\Psi_1} \right)^{1/(2\tau)} \in (0, 2R)$.
Plugging this value $r^*$ into aboves inequality gives
\begin{align*}
&~ \int_{B_{2R}} \Vert \omega \Vert^2 \cdot (1 + \Vert \omega \Vert^2)^{-\tau} \cdot \left| \sum_{j=1}^n \alpha_j e^{i \omega^\top x_j} \right|^2 ~ \mathrm{d}\omega \\
&= (r^*)^2 \Psi_1 + (r^*)^{-2(\tau - 1)} \Psi_2 \\
&= \left( \frac{ (\tau-1) \Psi_2 }{\Psi_1} \right)^{1/\tau} \Psi_1 
+ \left( \frac{\Psi_1}{(\tau-1) \Psi_2 } \right)^{1-1/\tau} \Psi_2 \\
&= (\tau-1)^{1/\tau} \Psi_2^{1/\tau} \Psi_1^{1-1/\tau} 
+ \Psi_1^{1-1/\tau} \Psi_2^{1/\tau} (\tau-1)^{1/\tau-1} \\
&= \Psi_1^{1-1/\tau} \Psi_2^{1/\tau} \cdot \left( (\tau-1)^{1/\tau} + (\tau-1)^{1/\tau-1} \right) \\
&\leq 2 \Psi_1^{1-1/\tau} \Psi_2^{1/\tau}.
\end{align*}
The final inequality comes from upper bounding the $\tau$ dependent term. 
\item The integral $\int_{\R^d \setminus B_{2R}}$ can be estimated straightforward by using $\sin^2 \leq 1$ as
\begin{align*}
&~ \int_{\R^d \setminus B_{2R}} \sin \left( \frac{\omega^\top b}{2} \right)^2 \cdot (1 + \Vert \omega \Vert^2)^{-\tau} \cdot \left| \sum_{j=1}^n \alpha_j e^{i \omega^\top x_j} \right|^2 ~ \mathrm{d}\omega \\
\leq&~ \int_{\R^d \setminus B_{2R}} (1 + \Vert \omega \Vert^2)^{-\tau} \cdot \left| \sum_{j=1}^n \alpha_j e^{i \omega^\top x_j} \right|^2 ~ \mathrm{d}\omega \\
\leq&~ \varepsilon (2\pi)^{d/2} c_{\min} q_X^{2\tau-d} \cdot \Vert \alpha \Vert^2.
\end{align*}
The final inequality comes from the penultimate inequality in \cite[Proof of Theorem 6]{wenzel2024novel}. 
\end{itemize}
Plugging both estimates together (recalling the factor $\frac{\Vert b \Vert^2}{4}$ of the first term) yields
\begin{align*}
&~\int_{\R^d} (1+\Vert \omega \Vert^2)^{-\tau} \cdot \left| \sum_{j=1}^n \alpha_j e^{i \omega^\top x_j}  \right|^2 \cdot \sin \left( \frac{\omega^\top b}{2} \right)^2 ~ \mathrm{d}\omega \\
\leq&~ \frac{\Vert b \Vert^2}{2} \Psi_1^{1-1/\tau} \Psi_2^{1/\tau} + \varepsilon (2\pi)^{d/2} c_{\min} q_X^{2\tau-d} \cdot \Vert \alpha \Vert^2 \\ 
\leq&~ \left( \frac{\Vert b \Vert^2}{2} \left( \frac{\alpha^\top A \alpha}{\Vert \alpha \Vert^2} \right)^{1-1/\tau} \left( c_2 (2R)^d \right)^{1/\tau} + \varepsilon (2\pi)^{d/2} c_{\min} q_X^{2\tau-d} \right) \cdot \Vert \alpha \Vert^2 \\
=&~ \left( \frac{c_2^{1/\tau} \Vert b \Vert^2}{2}  \left( \frac{\alpha^\top A \alpha}{\Vert \alpha \Vert^2} \right)^{1-1/\tau} (2R)^{d/\tau} + \varepsilon (2\pi)^{d/2} (c_{\min} q_X^{2\tau-d})^{1-1/\tau} (c_{\min} q_X^{2\tau-d})^{1/\tau} \right) \cdot \Vert \alpha \Vert^2 \\
\leq&~ \left( \frac{c_2^{1/\tau} c_0^{d/\tau} \Vert b \Vert^2}{2} \left( \frac{\alpha^\top A \alpha}{\Vert \alpha \Vert^2} \right)^{1-1/\tau} q_X^{-d/\tau} + \varepsilon (2\pi)^{d/2} \left( \frac{\alpha^\top A \alpha}{\Vert \alpha \Vert^2} \right)^{1-1/\tau} c_{\min}^{1/\tau} q_X^{2-d/\tau} \right) \cdot \Vert \alpha \Vert^2 \\
=&~ \left( \frac{c_2^{1/\tau} c_0^{d/\tau} \Vert b \Vert^2}{2 q_X^2} + \varepsilon (2\pi)^{d/2} c_{\min}^{1/\tau} \right) \left( \frac{\alpha^\top A \alpha}{\Vert \alpha \Vert^2} \right)^{1-1/\tau} q_X^{2-d/\tau} \cdot \Vert \alpha \Vert^2.
\end{align*}
Using $\Vert b \Vert \leq \left( \frac{2(2\pi)^{d/2} c_{\min}^{1/\tau}}{c_2^{1/\tau} c_0^{d/\tau}} \right)^{1/2} \varepsilon^{1/2} \cdot q_X =: c_\varepsilon \cdot q_X$,
the prefactor can be estimated with $2\varepsilon (2\pi)^{d/2} c_{\min}^{1/\tau}$.
\end{proof}

%


\section{Symmetric kernel part (see Eq.~\eqref{eq:intermediate_eq3})}
\label{sec:symm_kernel_part}

We denote 
\begin{align*}
k(X+b, X)_+ := \frac{k(X+b, X) + k(X+b, X)^\top}{2}
\end{align*}
using $k(X+b, X)_{ij} = k(x_i + b, x_j)$.
Then we compute
\begin{align*}
&~ (2\pi)^{d/2} \cdot \langle k(X+b, X)_+ \alpha, \alpha \rangle \\
=& ~ \frac{(2\pi)^{d/2}}{2} \cdot \left( \sum_{i,j=1}^n \alpha_i \alpha_j k(x_i + b, x_j) + \sum_{i,j=1}^n \alpha_i \alpha_j k(x_i, x_j+b) \right) \\
=& ~ \frac{1}{2} \cdot \int_{\R^d} \hat{\Phi}(\omega) \cdot \sum_{i,j=1}^n \alpha_i \alpha_j  \left( e^{i(x_i + b - x_j)\omega} + e^{(x_i - x_j - b)\omega} \right) ~ \mathrm{d}\omega \\
=& ~ \frac{1}{2} \cdot \int_{\R^d} \hat{\Phi}(\omega) \cdot \sum_{i,j=1}^n \alpha_i \alpha_j e^{i\omega^\top(x_i - x_j)} \left( e^{i\omega^\top b} + e^{- i \omega^\top b} \right) ~ \mathrm{d}\omega \\
=& ~ \int_{\R^d} \hat{\Phi}(\omega) \cdot \sum_{i,j=1}^n \alpha_i \alpha_j e^{i\omega^\top (x_i - x_j)} ~ \mathrm{d}\omega \\
&\quad - \frac{1}{2} \cdot \int_{\R^d} \hat{\Phi}(\omega) \cdot \sum_{i,j=1}^n \alpha_i \alpha_j e^{i\omega^\top (x_i - x_j)} \left( 1 - e^{i\omega^\top b} + 1 - e^{- i \omega^\top b} \right) ~ \mathrm{d}\omega \\
%
=& ~ \int_{\R^d} \hat{\Phi}(\omega) \cdot \left| \sum_{i,j=1}^n \alpha_i e^{i\omega^\top x_i} \right|^2 ~ \mathrm{d}\omega \\
&\quad - \frac{1}{2} \cdot \int_{\R^d} \hat{\Phi}(\omega) \cdot \sum_{i,j=1}^n \alpha_i \alpha_j e^{i \omega^\top (x_i - x_j)} \left( (1 - e^{i\omega^\top b})(1 - e^{- i \omega^\top b}) \right) ~ \mathrm{d}\omega \\
=& ~ \int_{\R^d} \hat{\Phi}(\omega) \cdot \left| \sum_{i=1}^n \alpha_i e^{i\omega^\top x_i} \right|^2 ~ \mathrm{d}\omega - \frac{1}{2} \cdot \int_{\R^d} \hat{\Phi}(\omega) \cdot \left| \sum_{i=1}^n \alpha_i e^{i\omega^\top x_i} (1 - e^{i\omega^\top b}) \right|^2 ~ \mathrm{d}\omega \\
=& ~ \int_{\R^d} \hat{\Phi}(\omega) \cdot \left| \sum_{i=1}^n \alpha_i e^{i\omega^\top x_i} \right|^2 ~ \mathrm{d}\omega \notag - 2 \int_{\R^d} \hat{\Phi}(\omega) \cdot \left| \sum_{i=1}^n \alpha_i e^{i\omega^\top x_i} \right|^2 \sin \left( \frac{\omega^\top b}{2} \right)^2  ~ \mathrm{d}\omega
\end{align*}

\end{document}

%% file: Figures/heatmap_identity_2.tex
\begin{tikzpicture}

\definecolor{darkgray176}{RGB}{176,176,176}

\begin{axis}[
width=1.0\fwidth,
height=1.0\fwidth,
point meta max=0,
point meta min=-5,
tick align=outside,
tick pos=left,
x grid style={darkgray176},
xmin=-0.5, xmax=49.5,
xtick style={color=black},
y dir=reverse,
y grid style={darkgray176},
ymin=-0.5, ymax=49.5,
ytick style={color=black}
]
\addplot graphics [includegraphics cmd=\pgfimage,xmin=-0.5, xmax=49.5, ymin=49.5, ymax=-0.5] {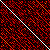};
\end{axis}

\end{tikzpicture}

%% file: Figures/heatmap_identity_3.tex
\begin{tikzpicture}

\definecolor{darkgray176}{RGB}{176,176,176}

\begin{axis}[
width=1.0\fwidth,
height=1.0\fwidth,
colorbar,
colorbar style={ylabel={}},
colormap={mymap}{[1pt]
  rgb(0pt)=(0.0416,0,0);
  rgb(365pt)=(1,0,0);
  rgb(746pt)=(1,1,0);
  rgb(1000pt)=(1,1,1)
},
point meta max=0,
point meta min=-5,
tick align=outside,
tick pos=left,
x grid style={darkgray176},
xmin=-0.5, xmax=49.5,
xtick style={color=black},
y dir=reverse,
y grid style={darkgray176},
ymin=-0.5, ymax=49.5,
ytick style={color=black}
]
\addplot graphics [includegraphics cmd=\pgfimage,xmin=-0.5, xmax=49.5, ymin=49.5, ymax=-0.5] {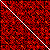};
\end{axis}

\end{tikzpicture}

%% file: Figures/min_eigenvalue_matern_0.tex
\begin{tikzpicture}

\definecolor{darkgray176}{RGB}{176,176,176}
\definecolor{darkorange25512714}{RGB}{255,127,14}
\definecolor{lightgray204}{RGB}{204,204,204}
\definecolor{steelblue31119180}{RGB}{31,119,180}

\begin{axis}[
width=0.951\fwidth,
height=0.75\fwidth,
legend cell align={left},
legend style={
  fill opacity=0.8,
  draw opacity=1,
  text opacity=1,
  at={(0.03,0.03)},
  anchor=south west,
  draw=lightgray204
},
log basis x={10},
log basis y={10},
tick align=outside,
tick pos=left,
x grid style={darkgray176},
xmin=8.60891659331735, xmax=232.317269928309,
xmode=log,
xtick style={color=black},
y grid style={darkgray176},
ymin=1.66665505038919e-09, ymax=0.129898574335127,
ymode=log,
ytick style={color=black},
xlabel={\#points}
]
\addplot [semithick, steelblue31119180]
table {%
10 0.0568706355670114
11 0.0509790108184866
12 0.0462028919765764
13 0.0422514746581091
15 0.0360910766583689
16 0.0336425094373854
18 0.0296274128354198
20 0.0264721339665011
22 0.0239264621735002
25 0.0209125978721527
28 0.0185747016811636
31 0.0167079271766553
34 0.0151827029214497
38 0.0135357861205502
42 0.0122115780074669
47 0.0108812797242231
52 0.00981255447584397
57 0.00893511542368145
64 0.00794112253319229
71 0.00714623205743548
78 0.0064960488316209
87 0.00581578333197956
97 0.00520965210649342
107 0.00471796274227804
119 0.00423800109328926
132 0.00381731586191683
146 0.00344866135030548
162 0.00310587206185355
180 0.00279350154815347
200 0.00251271251551893
};
\addplot [semithick, darkorange25512714]
table {%
10 0.00011886014854231
11 8.6252120083529e-05
12 6.45552621912302e-05
13 4.95643953872575e-05
15 3.10548270318192e-05
16 2.51978461772709e-05
18 1.72529274722496e-05
20 1.23266132363683e-05
22 9.1111548721959e-06
25 6.08967878326465e-06
28 4.26962252390914e-06
31 3.10845882520479e-06
34 2.33302598925687e-06
38 1.6534705822007e-06
42 1.21422930286145e-06
47 8.59109625445634e-07
52 6.3003277831556e-07
57 4.75683433737523e-07
64 3.3393208945073e-07
71 2.43352635791937e-07
78 1.82786546806138e-07
87 1.31162849704271e-07
97 9.42762588211258e-08
107 7.00214604715077e-08
119 5.07507370304113e-08
132 3.70873697152746e-08
146 2.73462015184235e-08
162 1.99751613844813e-08
180 1.45338712965351e-08
200 1.05769074353314e-08
};
\addplot [semithick, black, dashed]
table {%
10 0.0222222222222222
11 0.02
12 0.0181818181818182
13 0.0166666666666667
15 0.0142857142857143
16 0.0133333333333333
18 0.0117647058823529
20 0.0105263157894737
22 0.00952380952380952
25 0.00833333333333333
28 0.00740740740740741
31 0.00666666666666667
34 0.00606060606060606
38 0.00540540540540541
42 0.00487804878048781
47 0.00434782608695652
52 0.00392156862745098
57 0.00357142857142857
64 0.00317460317460317
71 0.00285714285714286
78 0.0025974025974026
87 0.00232558139534884
97 0.00208333333333333
107 0.00188679245283019
119 0.00169491525423729
132 0.00152671755725191
146 0.00137931034482759
162 0.00124223602484472
180 0.00111731843575419
200 0.00100502512562814
};
\addplot [semithick, black, dashed]
table {%
10 4.11522633744856e-05
11 3e-05
12 2.25394440270473e-05
13 1.73611111111111e-05
15 1.0932944606414e-05
16 8.88888888888889e-06
18 6.10624872786485e-06
20 4.37381542498906e-06
22 3.23939099449303e-06
25 2.17013888888889e-06
28 1.52415790275873e-06
31 1.11111111111111e-06
34 8.34794223223975e-07
38 5.9226501885377e-07
42 4.35280973868632e-07
47 3.08210733952494e-07
52 2.26157360291291e-07
57 1.70827259475219e-07
64 1.19977444240483e-07
71 8.74635568513119e-08
78 6.57126648018873e-08
87 4.71656583697033e-08
97 3.39084201388889e-08
107 2.51885784909691e-08
119 1.82589261803787e-08
132 1.33446555321826e-08
146 9.84050186559514e-09
162 7.18858854699695e-09
180 5.23072829696728e-09
200 3.806817222904e-09
};
\end{axis}

\end{tikzpicture}

%% file: Figures/min_eigenvalue_matern_1.tex
\begin{tikzpicture}

\definecolor{darkgray176}{RGB}{176,176,176}
\definecolor{darkorange25512714}{RGB}{255,127,14}
\definecolor{lightgray204}{RGB}{204,204,204}
\definecolor{steelblue31119180}{RGB}{31,119,180}

\begin{axis}[
width=0.951\fwidth,
height=0.75\fwidth,
legend cell align={left},
legend style={
  fill opacity=0.8,
  draw opacity=1,
  text opacity=1,
  at={(0.03,0.03)},
  anchor=south west,
  draw=lightgray204
},
log basis x={10},
log basis y={10},
tick align=outside,
tick pos=left,
x grid style={darkgray176},
xmin=8.60891659331735, xmax=232.317269928309,
xmode=log,
xtick style={color=black},
y grid style={darkgray176},
ymin=9.67108864420744e-21, ymax=0.000747825842771531,
ymode=log,
ytick style={color=black},
xlabel={\#points}
]
\addplot [semithick, steelblue31119180]
table {%
10 0.000127687777536716
11 9.12427051158741e-05
12 6.75264875039512e-05
13 5.14109596502163e-05
15 3.18457204284159e-05
16 2.57377980314241e-05
18 1.75221725720661e-05
20 1.24711574202464e-05
22 9.19350452871469e-06
25 6.12839311110785e-06
28 4.28944269243446e-06
31 3.11931750086371e-06
34 2.33931269898752e-06
38 1.65672429514872e-06
42 1.21602712880058e-06
47 8.60032058251325e-07
52 6.30538841431724e-07
57 4.75976682080733e-07
64 3.34079323271038e-07
71 2.43432004000361e-07
78 1.82831875424865e-07
87 1.31186490493603e-07
97 9.42886125714892e-08
107 7.00283369252972e-08
119 5.07543828796021e-08
132 3.70893328995976e-08
146 2.73472767741838e-08
162 1.99757381084406e-08
180 1.45341794331824e-08
200 1.05770702379763e-08
};
\addlegendentry{ { \scriptsize $\lambda_{\min}(k(X,X))$ } }
\addplot [semithick, darkorange25512714]
table {%
10 9.39015888450254e-10
11 4.29869620858017e-10
12 2.13088517923071e-10
13 1.1271358878666e-10
15 3.67643184905821e-11
16 2.23320874303429e-11
18 9.07610130292826e-12
20 4.09263205106748e-12
22 2.00428895930304e-12
25 7.75366265057479e-13
28 3.36430386937811e-13
31 1.59828618366265e-13
34 8.16443554393366e-14
38 3.65129080571313e-14
42 1.76747386387524e-14
47 7.7012745887476e-15
52 3.3827039678878e-15
57 1.43036537751662e-15
64 5.52982599028024e-16
71 -3.44371031809238e-16
78 -8.19750483133154e-16
87 -1.70476515936843e-15
97 -1.79321422436018e-15
107 -1.67655059790201e-15
119 -4.73624324452307e-15
132 -4.45740485548869e-15
146 -4.18317209979165e-15
162 -6.52022877316837e-15
180 -7.44260733160549e-15
200 -6.96632307535098e-15
};
\addlegendentry{ { \scriptsize $\lambda_{\min}(k^*(X,X))$ } }
\addplot [semithick, black, dashed]
table {%
10 2.74348422496571e-05
11 2e-05
12 1.50262960180316e-05
13 1.15740740740741e-05
15 7.28862973760933e-06
16 5.92592592592593e-06
18 4.07083248524323e-06
20 2.91587694999271e-06
22 2.15959399632869e-06
25 1.44675925925926e-06
28 1.01610526850582e-06
31 7.40740740740741e-07
34 5.56529482149317e-07
38 3.94843345902513e-07
42 2.90187315912421e-07
47 2.05473822634996e-07
52 1.50771573527527e-07
57 1.13884839650146e-07
64 7.99849628269885e-08
71 5.83090379008746e-08
78 4.38084432012582e-08
87 3.14437722464689e-08
97 2.26056134259259e-08
107 1.6792385660646e-08
119 1.21726174535858e-08
132 8.89643702145509e-09
146 6.56033457706343e-09
162 4.79239236466464e-09
180 3.48715219797818e-09
200 2.53787814860266e-09
};
\addplot [semithick, black, dashed]
table {%
10 1.46352610690138e-10
11 7e-11
12 3.59210682761495e-11
13 1.95357153063557e-11
15 6.64051543149538e-12
16 4.09693644261545e-12
18 1.70590793912325e-12
20 7.83110114638046e-13
22 3.8865385241491e-13
25 1.52622775830904e-13
28 6.69193464518236e-14
31 3.20073159579332e-14
34 1.64248140265887e-14
38 7.37370861232731e-15
42 3.59427285497066e-15
47 1.60617631803079e-15
52 7.80021920038067e-16
57 4.05304896941857e-16
64 1.77710952178743e-16
71 8.49985975231409e-17
78 4.36177203572242e-17
87 2.01191153826007e-17
97 9.31535497014796e-18
107 4.65539979535635e-18
119 2.19747523241397e-18
132 1.05730179720805e-18
146 5.19424014381204e-19
162 2.49641571487999e-19
180 1.18884934585263e-19
200 5.66404252262364e-20
};
\end{axis}

\end{tikzpicture}